\documentclass{article}

\usepackage{amsmath}
\usepackage{amssymb}
\usepackage{amsthm}
\usepackage[a4paper]{geometry}
\usepackage{graphicx}
\usepackage{tikz}
\usetikzlibrary{shapes.geometric,arrows}
\usepackage{setspace}
\usepackage{enumerate}
\numberwithin{equation}{section}
\newtheorem{thm}{Theorem}[section]
\newtheorem{lemma}[thm]{Lemma}
\newtheorem{prop}[thm]{Proposition}

\newtheorem{definition}[thm]{Definition}
\newtheorem{example}[thm]{Example}

\newtheorem{note}[thm]{Note}

\title{Ideal on CL-algebra}
\author{Safiqul Islam\\
	Department of Mathematics, Taki Government College, West Bengal, India\\
	safiqulwbes@gmail.com}

\begin{document}
	
	\maketitle
	
	\begin{abstract}
		In this paper, we introduce the concept of ideal on CL-algebra. It is proved that this concept generalizes the notion of ideal on Residuated Lattices. Prime ideal on CL-algebra are defined and few interesting properties are obtained. It has been shown that quotient algebra corresponding to CL-algebra is formed with the help of ideal is also a CL-algebra.
	\end{abstract}



\section{Introduction}

Ideal theory plays important role in the study of algebraic structures and associated logics, in some cases\cite{ras}. Recently, various researchers work on filters of various algebraic structures. P. H\'ajek introduced the concept of BL-algebra and also introduced the concept of filter on BL-algebra\cite{haj}. After that many authors contributed on various types of filters of BL-algebra\cite{tur1,tur2,sae1,sae2,sae3,sae4,mog,mot,kon}. Filters on other algbraic structures are also available in literature\cite{bor}. Although there are so many papers on filters of BL-algebra, the notion of ideals is missing in BL-algebras for
lack of a suitable algebraic addition. A solution to this problem is given at \cite{lele}. There is a subsequent work on ideals of BL-algebra in \cite{yang}.

Classical Linear Algebra (CL-algebra, in short) were introduced by A. Troelstra\cite{tro} as an algebraic counterpart of linear logic\cite{gir}. CL-algebra can be found as $FL_e$-algbra with involution in \cite{ono}. Various properties of CL-algebra are studied in \cite{senth,sen}. Ideal on CL-algebra is not available in literature. In this paper, we introduced the concept of CL-algebra. Properties of ideals on CL-algebra are studied here.

In the next section, definition of CL-algebra is given. Properties and examples of CL-algebra are also discussed in this section. In section 3, ideal on CL-algebra is introduced. Properties of ideals on CL-algebra are obtained in this section. With the help of ideal, we define a congruence relation on an arbitrary CL-algebra and proved that the corresponding quotient algebra is also a CL-algebra with respect to suitable operations. Some special type of ideals on CL-algebra are discussed in Section 4.

\section{Clasical Linear Algebra}
\begin{definition}
	Let $L$ be a non empty set. A Clasical linear algebra (CL-algebra, in short) \cite{tro} is an algebraic system $L=(L,\cup ,\cap, \bot, \to, \ast,0,1 )$  which satisfy the following conditions:
	\begin{enumerate}
		\item $(L,\cup,\cap,\bot)$ is a lattice with least element $\bot$
		\item $(L, \ast, 1)$ is a commutative monoid with unit 1.
		\item for any $x,y,z \in L$ , $x\ast y \leq z$ if and only if $x \leq y \to z$ [residuation property]
		\item $\sim \sim x=x$, for all $x\in L$ where $\sim x=x\to 0.$ [involution]
	\end{enumerate}
	
\end{definition}

\begin{prop}\label{lem1}
	In every CL-algebra $L$ the following properties hold for all $x,y,x_1, y_1$ in $L$. \cite{tro,senth,sen}
	\begin{enumerate}
		\item $x \ast (y \cup z)= (x \ast y)\cup (x \ast z)$ and moreover, if the join $\underset{i\in I}{\bigcup}y_i$ exists, then $x \ast \underset{i\in I}{\bigcup}y_i= \underset{i\in I}{\bigcup}(x \ast y_i)$. 
		\item $\bot \to \bot$ is the largest element of $L$ and is denoted by $\top$.
		\item If $x, y \leq 1$  then $x \ast y  \leq x\cap y$.
		\item $1 \leq x, y$ then $x \cup y \leq x  \ast y$.
		\item $(x\rightarrow y)\ast (y\rightarrow z)\leq (x\rightarrow z)$.
		\item $1\to x =x$.
		\item If $x\leq x_1, y\leq y_1$ then $x\ast y\leq x_1\ast y_1$ and $x_1\to y\leq x\to y_1$.
		\item $x\to (y\to z) = (x \ast y)\to z$.
		\item $x\ast (x\to y) \leq y$.
		\item If $x\leq y$ then $\sim x\leq \sim y.$
		\item $x\cup y=\sim(\sim x\cap \sim y)$.
		\item $x\cap y=\sim(\sim x\cup \sim y)$.
		\item $x\to y=\sim(x \ast \sim y)$.
		\item $\sim x\to y=\sim(\sim x\ast \sim y)$. We denote $\sim(\sim x\ast \sim y)$ by $x+y$.
		\item $\sim \top =\bot$.
		\item $\sim \top \ast \top = \bot$.
	
\end{enumerate}
\end{prop}
\begin{example} \label{ex1}
	Example of a linear CL-algebra is given here.\\
	Let $L=\{\bot,0,a,1,\top\}$. The lattice is shown below. 
	\begin{center}
		\begin{tikzpicture}[scale=1]
		\node (top) at (0,4) {$\top$}; 
		\node (one) at (0,3) {$1$}; 
		\node (a) at (0,2) {$a$}; 
		\node (zero) at (0,1) {$0$};
		\node (bot) at (0,0) {$\bot$};
		\draw (bot) -- (zero) -- (a) -- (one) -- (top);
		\end{tikzpicture}
	\end{center}
	We define $\ast, \to$ as follows\\	
	\begin{tabular}{c|ccccc}
		$\ast$  & $\bot$ & $0$ & $1$ & $a$ & $\top$  \\
		\hline
		$\bot$ & $\bot$ & $\bot$ & $\bot$ & $\bot$ & $\bot$   \\
		$0$ & $\bot$   & $0$ & $0$ & $0$ & $\top$ \\
		$1$ & $\bot$   & $0$ & $1$ & $a$& $\top$ \\
		$a$& $\bot$ & $0$ & $a$& $0$ & $\top$  \\
		$\top$ & $\bot$ & $\top$& $\top$& $\top$ & $\top$		
	\end{tabular} 
\hspace{2cm}
\begin{tabular}{c|ccccc}
	$\to$  & $\bot$ & $0$ & $1$ & $a$ & $\top$  \\
	\hline
	$\bot$ & $\top$ & $\top$ & $\top$ & $\top$ & $\top$   \\
	$0$ & $\bot$   & $1$ & $1$ & $1$ & $\top$ \\
	$1$ & $\bot$   & $0$ & $1$ & $a$& $\top$ \\
	$a$& $\bot$ & $a$ & $1$ & $1$ & $\top$  \\
	$\top$ & $\bot$ & $\bot$& $\bot$& $\bot$ & $\top$	
\end{tabular}
\end{example}

\begin{example} \label{ex2}
	Example of a non-linear CL-algebra is the following.\\
	Let $L=\{\bot,0,a,b,1,\top\}$. The lattice is shown below.
	
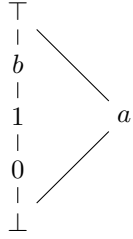
\begin{center}
		\begin{tikzpicture}[scale=.7]
		\node (top) at (0,2) {$\top$}; 
		\node (a) at (2,0) {$a$}; 
		\node (one) at (0,0) {$1$}; 
		\node (b) at (0,1) {$b$};
		\node (zero) at (0,-1) {$0$}; 
		\node (bot) at (0,-2) {$\bot$};
		\draw (bot) -- (zero) -- (one) -- (b) -- (top);
		\draw (bot) -- (a) -- (top);
		\end{tikzpicture}
	\end{center} 
	We define $\ast, \to$ as follows\\
\begin{tabular}{c|c c c c c c }
	$\ast$ & $\bot$ & $0$ & $1$ & $a$ & $b$ & $\top$  \\
	\hline
	$\bot$ & $\bot$ & $\bot$ & $\bot$ & $\bot$ & $\bot$ & $\bot$   \\
	$0$ & $\bot$ & $0$ & $0$ & $a$ & $0$ & $\top$ \\
	$1$ & $\bot$ & $0$ & $1$ & $a$ & $b$ & $\top$ \\
	$a$ & $\bot$ & $a$ & $a$ & $0$ & $a$ & $\top$ \\
	$b$ & $\bot$ & $0$ & $b$ & $a$ & $0$ & $\top$\\ 
	$\top$ & $\bot$ & $\top$ & $\top$ & $\top$ & $\top$ & $\top$	
\end{tabular}
\hspace{2cm}
\begin{tabular}{c|c c c c c c }
	$\rightarrow$  & $\bot$ & $0$ & $1$ & $a$ & $b$ & $\top$  \\
	\hline
	$\bot$ & $\top$ & $\top$ & $\top$ & $\top$ & $\top$ & $\top$   \\
	$0$ & $\bot$ & $1$ & $1$ & $a$ & $1$ & $\top$ \\
	$1$ & $\bot$ & $0$ & $1$ & $a$ & $b$ & $\top$ \\
	$a$ & $\bot$ & $a$ & $a$ & $1$ & $a$ & $\top$ \\
	$b$ & $\bot$ & $b$ & $1$ & $a$ & $1$ & $\top$\\
	$\top$ & $\bot$ & $\bot$ & $\bot$ & $\bot$ & $\bot$ & $\top$	
\end{tabular}\\
\end{example}

\section{Ideal}
Ideal on CL-algebra is introduced here.
\begin{definition}
	$L=(L,\cup ,\cap, \bot, \to, \ast,0,1 )$ be a CL-algebra and $I$ be a non empty subset of $L$. Then $I$ is called an ideal if the following holds 
	\begin{enumerate}
		\item $0\in I$.
		\item $x,y\in I$ then $x + y\in I$  i.e., $\sim  x\to y\in I $ i.e., $\sim(\sim x\ast \sim y)\in I$ and $x\cup y\in I$.
		\item If $y \in I$ and $x\leq y$ then $x \in I$.
		\end{enumerate}
\end{definition}
\begin{example}
	In example \ref{ex2} if we take $I=\{\bot,0,1,b \}$ then $I$ is an ideal of $L.$\\
	Also $I=\{\bot,0,1 \}$ is another ideal of $L$.
	
\end{example}
\begin{prop} \label{prop1}
	Let $x\leq y$ in $L$ then $\sim (x \to y)\in I$.
	\begin{proof}
		Let $x\leq y$ then  $1\ast x\leq y$, so $1\leq x\to y$.\\
		therefore $\sim(x\to y)\leq \sim1=0$ as $0\in I$,  then by definition of ideal,  $ \sim(x\to y)\in I$.
	\end{proof}
\end{prop}
\begin{thm}
	Let $I$ be an ideal in  a CL-algebra $L$. Define a binary relation $\rho$ on $L$ by $x \rho y$  if and only if $x \ast \sim y \in I$ and $y \ast \sim x \in I$
	then  $\rho$ is a congruence relation on $L$. The set of all congruence  classes is denoted by $L/I$ i.e., \\
	$L/I = {\{[x]: x \in L\}} $ where $[x]= \{y \in L : x \rho y \}$.
	
\end{thm}
Now define $[x] \ \Box\ [y]$ = $ [x \Box y]$ where $\Box \in \{\cap, \cup, \to \ast\}$ and $\sim [x]= [\sim x]$.
\begin{prop} \label{prop2}
$[x]\leq [y]$ if and only if $\sim (x\to y) \in I$.
\end{prop}
\begin{proof}
	Let $[x]\leq [y]$ then $[x]\cap [y]=[x]$ i.e., $[x\cap y]=[x]$.
Then by definition of $\rho$,  $(x\cap y)\ast \sim x,$ $\sim(x\cap y)\ast x$ both are in $I$.\\
Now $x \cap y\leq y$ therefore $\sim((x\cap y)\to y)\in I$. \\
Also $\sim (x\to (x\cap y))\in I$. Again we have 
$(x\to (x\cap y))\ast ((x\cap y)\to y)\leq (x\to y)$ so $\sim (x\to y)\leq \sim ( ( x\to (x\cap y)\ast ((x\cap y)\to y))\\
=\sim  ( x\to (x\cap y))+ \sim ((x\cap y)\to y)\in I $, so $\sim (x \to y)\in I$.\\

Conversely assume that $\sim(x\to y)\in I$.\\
Again $\sim(x\to x)\leq 0$, therefore $\sim(x\to x) \cup \sim (x\to y)\in I$.\\
Hence $\sim\{(x\to x)\cap(x\to y)\}\in I$.\\
Again $x\ast\{(x\to x)\cap (x\to y)\}\leq x\ast (x\to y)\leq y$ by Proposition \ref{lem1}.9\\
and $x\ast\{(x\to x)\cap (x\to y)\}\leq x\ast (x\to x)\leq x$.\\
then $x\ast\{(x\to x)\cap (x\to y)\}\leq (x\cap y)$, \\
therefore $\{(x\to x)\cap (x\to y)\}\leq(x \to  (x\cap y)$.\\
Thus $\sim (x \to  (x\cap y)\leq \sim\{(x\to x)\cap(x\to y)\}$\\
Therefore $\sim (x \to  (x\cap y) \in I$ or $x\ast \sim (x\cap y)\in I$......(i)\\
Again $(x\cap y)\leq x$ implies $\sim ((x\cap y)\to x)\in I$ or $(x\cap y)\ast \sim x\in I$.....(ii)\\
From (i) and (ii) we get $[x]\leq [y]$.
\end{proof}
\begin{lemma}\label{lem2}
	In a CL-algebra $L$,  $(z\to x) \cap (z \to y)= (z)\to (x \cap y)$ for all $x,y,z \in L$.
\end{lemma}
\begin{proof}
	As $x \cap y \leq x,y$, by Theorem \ref{lem1}.7, we get $z\to (x \cap y)\leq z \to x, z\to y$. \\ 
	Therefore $z\to (x\cap y)\leq (z\to x)\cap(z\to y)$. \\ 
	Now  $z\ast\{(z \to x)\cap(z \to y)\} \leq z\ast (z\to x)\leq x $. \\
	Similarly, $z\ast\{(z \to x)\cap(z \to y)\} \leq y $.
	So $z\ast \{(z \to x)\cap (z \to y)\} \leq x \cap y$. \\
	Then by residuation property, we get $\{(z \to x)\cap (z \to y)\} \leq z\to  (x \cap y)$. \\	
	Hence $\{(z \to x)\cap (z \to y)\} = z\to  (x \cap y)$. 	
\end{proof}
\begin{prop}
	$L/I$ is closed under $\Box$, where $\Box \in \{\cap, \cup, \to \ast, \sim\}$.
\end{prop}

\begin{proof} We first show that $L/I$ is closed under $\cap$.\\
To show that $[x\cap y]=[x] \cap [y]$\\
let, $x_1\in[x] , y_1\in[y]$\\
Now $[x\cap y]=[x_1\cap y_1]$.Therefore $\sim (x\cap y)\ast (x_1\cap y_1)\in I$ and $(x\cap y)\ast \sim(x_1\cap y_1)\in I$\\
Also $\sim x_1\ast x$, $x_1\ast \sim x$, $y_1\ast \sim y$, $y\ast \sim y$ $ \in I$.\\
By Proposition \ref{lem1}.7, $x\to x_1\leq ((x\cap y)\to x_1)$ therefore $\sim ((x\cap y)\to x_1)\leq \sim(x\to x_1)= (x\ast \sim x_1)$
, since $(x\ast \sim x_1) \in I$ i,e $(x\cap y)\ast \sim x_1 \in I$.\\
Similarly $(x\cap y)\ast\sim y_1\in I$. \\
Then by definition of ideal $((x\cap y)\ast \sim x_1 )\cup ((x\cap y)\ast\sim y_1) \in I$     ........(i).\\
 by Lemma \ref{lem2} $\sim (((x\cap y) \ast \sim x_1)\cap ((x\cap y) \ast  \sim y_1)) =\sim ((x\cap y)\ast \sim (x_1 \cap y_1))$\\
or, $\sim ((((x\cap y) \ast \sim x_1)\cup ((x\cap y) \ast \sim y_1))) =\sim ((x\cap y)\ast\sim  (x_1 \cap y_1))$\\
or,$(((x\cap y) \ast \sim x_1)\cup ((x\cap y) \ast \sim y_1)) =((x\cap y)\ast \sim  (x_1 \cap y_1))$\\
then by (i) $((x\cap y)\ast \sim  (x_1 \cap y_1))\in I$\\
Similarly by changing the roles of $x, y$ with $x_1, y_1$ respectively we get $(\sim (x\cap y)\ast  (x_1 \cap y_1))\in I$
Then  $[x\cap y]= [(x_1 \cap y_1)]$.
\end{proof}
By the above procedure we also prove that  $L/I$ closed with respect to $\ast,\cup,\to,\sim$. 
\begin{thm}
$(L/I, \cap, \cup,[\bot]\ast, \to ,[0], [1])$ is a CL-algebra with respect to the operations defined by\\
$[x]\cup [y]=[x\cup y]$\\
$[x]\cap [y]=[x\cap y]$\\
$[x]\ast [y]=[x\ast y]$\\
$[x]\to [y]=[x\to y]$.\\
$\sim [x]=[\sim x]$
\end{thm}
\begin{proof}
\begin{itemize}
\item It is obvious that $L/I$ is a bounded lattice with least element $[\bot]$ and top element $[\top]$, where\\
$[\bot]=\{x\in L:x\ast \top\in I\}$ and $[\top]=\{x\in L:\top \ast \sim x \in I\}$.
\item It is clear that $(L/I,\ast,[1])$ is a commutative moniod with $[1]$ as unit, where \\
$[1]=\{x\in L:x\ast \sim 1\in I, \sim x\ast 1\in I\}=\{x\in L:\sim x\in I ,\sim(x\to 1)\in I\}$
\item Let $[x],[y],[z] \in L/I$ with $[x]\ast [y]\leq [z]$ i.e., $[x\ast y]\leq [z]$.\\
Then by Proposition \ref{prop2}, $\sim ((x \ast y)\to z)\in I$.\\
Again by Proposition \ref{lem1}.8, $x\to (y \to z)=(x\ast y)\to z$, therefore \\
$\sim (x\to (y\to z))= \sim ((x\ast y)\to z) $, hence $\sim (x\to (y\to z)\in I$, this implies that $[x]\leq [y\to z].$\\
Conversely assume that $[x]\leq[y\to z]$, then $\sim (x\to (y\to z))\in I$. So, $\sim ((x\ast y)\to z)\in I$. Thus $[x\ast y]\leq [z]$.
\item Finally $\sim\sim[x]=\sim[\sim x]=[\sim \sim x]=[x]$. 
\end{itemize}
Hence $(L/I, \cap, \cup,[\bot]\ast, \to ,[0], [1])$ is a CL-algebra.
\end{proof}
\begin{thm}
If $I=\{x:x\leq 0\}$ then $[x]=\{x\}$.
\end{thm}
\begin{proof}
Let $x\in I$.\\
Assume that $y\in [x].$ then $x\ast \sim y$ and $\sim x\ast y$ both are in $I$.\\
Then by construction of $I$, $y\ast \sim x\leq 0$ or $\sim x\leq y\to 0=\sim y$or $y \leq x$.\\
Similarly $x\leq x$. Then by combining two we get $x=y.$
\end{proof}
\section{Special types of ideal}
\begin{definition}
An ideal $I$ of a CL-algebra $L$ is said to be distributive if for any $x,y,z\in L$ we have $((x\cup y)\cap (x\cup z))\ast \sim (x\cup (y\cap z))\in I $. 
\end{definition}
\begin{example}
		The ideal $I=\{\bot ,a, 0,1\}$ in Example \ref{ex1} is  a distributive  ideal.\\
\end{example}
\begin{thm}
If $I$ be a distributive ideal in a CL-algebra $L$ then $L/I$ is  a distributive CL-algebra.
\end{thm}
\begin{proof}
$x\leq x\cup y, x\cap z\leq x\cup y$ therefore $x\cup (y\cap z)\leq x\cup y$, similarly  $x\cup (y \cap z)\leq x\cup z$. Hence $x\cup (y\cap z)\leq (x\cup y)\cap (x\cup z)$ and so $x\cup (y\cap z)\ast \sim( (x\cup y)\cap (x\cup z))$.Therefore $[x]\cup ([y]\cap [z]) \leq ([x]\cup[y])\cap ([x]\cup [z])$.\par
On the other hand from the definition of distributive ideal we have $([x]\cup[y])\cap ([x]\cup [z]) \leq [x]\cup ([y]\cap [z]) $.Hence the theorem.
\end{proof}
\begin{definition}
An ideal $I$ is said to be prime if for any $x, y \in L$, $\sim (x\to y) \in I$ or $\sim (y\to x )\in I$
\end{definition}
\begin{example}
	Example of a prime ideal.\\
		The ideal $I=\{\bot ,a, 0,1\}$ in Example \ref{ex1} is  a prime ideal.\\
where as the ideal $I=\{\bot,0,1 \}$ in Example \ref{ex2} is not a prime ideal as $\sim (1\to a)=\sim a=a$ and $\sim (a\to 1)=\sim a=a$ both are not in $I$.
\end{example}
\begin{thm}
If $I$ be a prime ideal then $L/I$ is linear.
\end{thm}
\begin{proof}
Let $I$ be a prime ideal of $L$\\
Let $[x], [y] \in L/I.$ then by definition of prime ideal$\sim (x\to y) \in I$ or $\sim (y\to x )\in I$ so by Proposition \ref{prop2} $[x]\leq[y]$ or $[y]\leq [x]$.
\end{proof}

\begin{definition}
Let $L$ be a CL-algebra. A non empty subset $I$ of $L$ is said to be  implicative ideal if 
\begin{itemize}
\item $0\in I$
\item $\sim (x\to (y\to z)) \in I$ and $\sim (x\to y) \in I$ then $\sim (x\to z)\in I$.
\end{itemize}
\end{definition}
\begin{prop}
If every element of a CL-algebra is idempotent and  $I$ be a ideal in $L$.Then $I$ is an implicative ideal.
\end{prop}
\begin{proof}
Since $I$ is an ideal so $0\in I$.Let for all $x,y,z \in L$, $\sim (x\to (y\to z))\in I$ and $\sim (x\to y)\in I$.\\
Now $x\ast (x\to y)\leq y$ and $x\ast (x\to (y\to z))\leq (y \to z)$, hence $x\ast x\ast(x\to y)\ast (x\to (y\to z))\leq y\ast (y\to z)\leq z$,i.e., $(x\to y)\ast (x\to (y\to z))\leq x\to z$, so $\sim (x\to z)\leq \sim ((x\to y)\ast (x\to (y\to z))=\sim (x\to y) + \sim (x\to (y\to z))\in I$, therefore $\sim (x\to z) \in I$.Therefore $I$ is an implicative ideal.
\end{proof}
\begin{definition}
	An algebraic structure  $L=(L,\cup ,\cap, \bot, \to, \ast,1 )$ is said to be a residuated lattice if
\begin{itemize}
	\item $(L,\cup ,\cap, \bot, 1)$ is a bounded lattice.
	\item $(L,\ast, 1)$ is a commutative moniod.
	
	\item for any $x,y,z \in L$ , $x\ast y \leq z$ if and only if $x \leq y \to z$.
\end{itemize}
\end{definition}
\begin{prop}\label{rl}
	$x\ast y\leq x$ if and only if $\top=1.$
\end{prop}
\begin{proof}
	Let $x\ast y\leq x.$ Therefore $1\leq T$ or $x\leq x\ast T\leq x$ so $x\ast \top =x$ or$ 1\ast \top =1$ or $\top =1$.\\
	Conversely assume that $\top =1$, now $y\leq \top$ or $x\ast y\leq x\ast \top=x.$
\end{proof}
\begin{note}
	From the Proposition \ref{rl}, it follows that an CL-algebra $L$ satisfying $x\ast y\leq x$  for all $x,y\in L$ (or equivalently $\top=1$) is a residuated lattce. \\
	If we consider $\top=1$ in a CL-algebra, by Proposition \ref{lem1}.4 the condition `$x \cup y \in I$ for all $x,y\in I$' of the definition of ideal $I$ on CL-algebra becomes redundant. Thus the concept of ideal on CL-algebra generalizes the notion of ideal on Residuated Lattice.
\end{note}
\begin{definition}
An ideal $I$ of a CL-algebra $L$ is said to be an affine ideal if $\top \ast 0\in I.$
\end{definition}
\begin{thm}
Let $I$ be an affine ideal of a CL-algebra $L$ then the quotient algebra $L/I$ is a residuated lattice. 
\end{thm}
\begin{proof}
 We have $1\leq \top$ then by residuation property, $1\leq 1\to \top$. So, $\sim (1\to \top)\in I $ i.e., $1\ast \sim \top \in I$.\\
 Again by definition of affine ideal we have $\top \ast 0 =\top \ast \sim 1 \in I$. So $[1]=[\top]$. 
 Therefore $L/I$ is a residuated lattice.
\end{proof}

\section{Conclusion}
We introduce the concept of ideal on CL-algebra in this paper. Some other algebraic structures closely related to CL-agebra, namely ACL-algebra, IL-algebra, ILZ-algebra, are not considered here. Ideal on these structures may be taken into consideration in future. Further properties on ideals of CL-algebra may also be investigated. This paper is a first step in this new area of research. It will open scope of work on multiple areas, study of ideals on algebraic structures related with linear logic, reflection of algebraic results based on ideals in corresponding logic and many more.

\section{Acknowledgment}
The last two authors acknowledge Department of Higher Education, Science \& Technology and Bio-Technology, Government of West Bengal, India for the financial support in the research project 257(Sanc)/ST/P/S\&T/16G-45/2017 dated 25.03.2018.

\end{document}